\documentclass[11pt, reqno]{amsart}
\usepackage{amsmath, amsfonts, amsthm, amssymb, multicol, mathtools, dsfont,verbatim}
\usepackage{graphicx}
\usepackage{float, hyperref}
\usepackage{scalerel,stackengine, subcaption}
\usepackage[usenames,dvipsnames,x11names]{xcolor}
\usepackage{enumitem}
\usepackage{pgfplots}
\usepgfplotslibrary{fillbetween}
\pgfplotsset{width=10cm,compat=1.9}
\usepackage[square,sort,comma,numbers]{natbib}
\makeatletter
\def\@setauthors{%
  \begingroup
  \def\thanks{\protect\thanks@warning}%
  \trivlist
  \centering\footnotesize \@topsep30\p@\relax
  \advance\@topsep by -\baselineskip
  \item\relax
  \author@andify\authors
  \def\\{\protect\linebreak}

  \normalsize\lowercase{\authors}%
  
	\ifx\@empty\contribs
  \else
    ,\penalty-3 \space \@setcontribs
    \@closetoccontribs
  \fi
  \endtrivlist
  \endgroup
}
\def\@settitle{\begin{center}
\LARGE\lowercase{\@title}
  \end{center}%
}
\makeatother
\newcommand{\authoremail}[1]{\email{\href{mailto:#1}{\color{lightblue}{#1}}}}
\newcommand{\authoraddress}[1]{\address{\normalfont{#1}}}

\usepackage{fancyhdr}
\numberwithin{equation}{section}

\newtheorem{thm}{Theorem}[section]
\newtheorem{lma}[thm]{Lemma}
\newtheorem{cor}[thm]{Corollary}

\newtheorem{prop}[thm]{Proposition}

\newtheorem{ques}[thm]{Question}

\renewcommand{\epsilon}{\varepsilon}
\newcommand{\eps}{\varepsilon}
\newcommand{\rr}{\mathbb{R}^2}
\newcommand{\rd}{\mathbb{R}^d}
\newcommand{\zd}{\mathbb{Z}^d}

\renewcommand{\geq}{\geqslant}
\renewcommand{\leq}{\leqslant}

\newcommand{\hd}{\dim_{\textup{H}}}

\newcommand{\fs}{\dim^\theta_{\mathrm{F}}}
\newcommand{\fd}{\dim_{\mathrm{F}}}
\newcommand{\sd}{\dim_{\mathrm{S}}}

\newcommand{\gdk}{G(d,k)}
\newcommand{\gdo}{G(d,1)}
\newcommand{\gto}{G(2,1)}
\newcommand{\ddk}{d\gamma_{d,k}(V)}
\newcommand{\dhky}{d\H^k(y)}
\newcommand{\gamdk}{\gamma_{d,k}}

\newcommand{\R}{\mathbb{R}}
\newcommand{\rk}{\mathbb{R}^k}

\newcommand{\Z}{\mathbb{Z}}
\renewcommand{\H}{\mathcal{H}}

\newcommand{\N}{\mathbb{N}}
\renewcommand{\L}{\mathcal{L}}
\newcommand{\M}{\mathcal{M}}
\renewcommand{\S}{\mathcal{S}}

\newcommand{\spt}{\text{spt}\,}

\newcommand{\J}{\mathcal{J}}

\newcommand{\half}{\frac{1}{2}}

\makeatletter
\DeclareRobustCommand\widecheck[1]{{\mathpalette\@widecheck{#1}}}
\def\@widecheck#1#2{%
    \setbox\z@\hbox{\m@th$#1#2$}%
    \setbox\tw@\hbox{\m@th$#1%
       \widehat{%
          \vrule\@width\z@\@height\ht\z@
          \vrule\@height\z@\@width\wd\z@}$}%
    \dp\tw@-\ht\z@
    \@tempdima\ht\z@ \advance\@tempdima2\ht\tw@ \divide\@tempdima\thr@@
    \setbox\tw@\hbox{%
       \raise\@tempdima\hbox{\scalebox{1}[-1]{\lower\@tempdima\box
\tw@}}}%
    {\ooalign{\box\tw@ \cr \box\z@}}}
\makeatother

\stackMath
\newcommand\reallywidehat[1]{%
\savestack{\tmpbox}{\stretchto{%
  \scaleto{%
    \scalerel*[\widthof{\ensuremath{#1}}]{\kern.1pt\mathchar"0362\kern.1pt}%
    {\rule{0ex}{\textheight}}
  }{\textheight}%
}{2.4ex}}%
\stackon[-6.9pt]{#1}{\tmpbox}%
}
\definecolor{lightblue}{HTML}{2B77A4}
\colorlet{plotblue}{LightSkyBlue3!80}
\definecolor{darkred}{HTML}{9E0D0D}
\definecolor{darkyellow}{HTML}{b3b300}
\definecolor{darkorange}{HTML}{D86129}
\hypersetup{
	colorlinks=true,
	linkcolor=darkred,
	urlcolor=darkred,
	citecolor=lightblue
}
\newcommand\numberthis{\addtocounter{equation}{1}\tag{\theequation}}

\title{A Fourier analytic approach to exceptional set\\estimates for orthogonal projections}

\author{Jonathan M. Fraser}
\authoraddress{J. M. Fraser, University of St Andrews, Scotland}
\authoremail{jmf32@st-andrews.ac.uk}
\thanks{JMF was financially supported by  a  \emph{Leverhulme Trust Research Project Grant} (RPG-2019-034).}

\author{Ana E. de Orellana}
\authoraddress{A. E. de Orellana, University of St Andrews, Scotland}
\authoremail{aedo1@st-andrews.ac.uk}
\thanks{AEdO was financially supported by the University of St Andrews.}
\date{}

\begin{document}
\maketitle
\thispagestyle{empty}

\begin{abstract}
  Marstrand's celebrated projection theorem gives the Hausdorff dimension of the orthogonal projection of a Borel set in Euclidean space for almost all orthogonal projections. It is straightforward to see that sets for which the Fourier and Hausdorff dimension coincide have no exceptional  projections, that is, \emph{all} orthogonal projections satisfy the conclusion of Marstrand's theorem. With this in mind, one might believe that the Fourier dimension (or at least, Fourier decay) could be used to give better estimates for the Hausdorff dimension of the exceptional set in general. We obtain projection theorems and exceptional set estimates based on the Fourier spectrum; a family of dimensions that interpolates between the Fourier and Hausdorff dimensions. We apply these results to show that the Fourier spectrum can be used to improve several results for the Hausdorff dimension in certain cases, such as Ren--Wang's sharp bound for the exceptional set in the plane, Peres--Schlag's exceptional set bound and Bourgain--Oberlin's sharp $0$-dimensional exceptional set estimate.
\end{abstract}

\tableofcontents

\section{Introduction}

Orthogonal projections are among the most studied objects in Fractal Geometry. Although Marstrand's work in \cite{Mar54} showed that the orthogonal projection of a Borel set is typically as big as possible, there can be a large set of directions for which this does not hold. This immediately gives rise to the problem of bounding from above the dimension of this set of exceptions.  Since early work of Kaufman \cite{Kau68}, a lot of work has been devoted to this topic. However, after almost 60 years and in spite of  recent and significant progress, the picture is far from being complete. We refer the reader to \cite{FFJ15,Mat14} for a more thorough presentation on the work of Marstrand and its consequences in Fractal Geometry. A short summary of recent research will be given in Section~\ref{section:preliminariesProj}.

There are certain specific situations in  which one can improve on Marstrand's theorem. For example, it is easy to see that there are no exceptional projections for Salem sets, that is,  sets that have the same Fourier and Hausdorff dimension. Given this observation, one might expect that if the Fourier dimension of a set is positive, but not necessarily equal to its Hausdorff dimension, one may be able to obtain improved estimates on the Hausdorff dimension of the exceptional set. Indeed, the Fourier dimension does give an easy bound that improves results for general sets; see the discussion at the start of Section \ref{section:FDproj} or, more concretely,  Proposition~\ref{prop:boundmuV} (set $\theta=0$ and use the bound $\hd P_V(X) \geq \fd P_V(X)$). However, we show in Proposition~\ref{propo:counter_example} that it does not do more than that. Therefore, the Fourier dimension is perhaps too coarse to capture the full power of Fourier decay in dealing with projections.

 Instead of studying the exceptional set using the Fourier and Hausdorff dimensions in isolation, we will benefit from the method of `dimension interpolation', which is designed to extract more nuanced  information from a given set or measure by considering a parametrised  family of dimensions living between two dimensions of interest. In this case, the family of dimensions that we will consider is  the Fourier spectrum, introduced  in \cite{Fra22}, which continuously interpolates between the Fourier and Hausdorff dimensions.

We begin by analysing the well-studied case of the exceptional set for the Hausdorff dimension of projections, giving in Lemma~\ref{lemma:sharp1} an example that establishes  the (well-known folklore) sharpness of the bound proven by Ren and Wang in \cite{RW23}. After studying the role of the Fourier dimension and its limitations in this setting, in Sections~\ref{section:FourierProj}  and \ref{section:exceptionalSpectrum} we obtain non-trivial bounds for the Fourier spectrum of projections (Theorem~\ref{thm:spectrumProj} and Proposition \ref{prop:boundmuV}) and give a bound for the Hausdorff dimension of the exceptional set of directions for the Fourier spectrum (Theorem ~\ref{thm:exceptionalFourier}). Despite the estimate from Theorem~\ref{thm:exceptionalFourier} not being sharp in general, we show in Section~\ref{section:applications} that it can be used to improve on sharp exceptional set estimates for the \emph{Hausdorff} dimension in certain cases: if the Fourier analytic information captured by the Fourier spectrum is `strong enough', then non-trivial improvements can be made to many exceptional set estimates, see Proposition~\ref{prop:improveOberlin} and Corollaries~\ref{thm:exceptionalFouriercoro} and \ref{thm:exceptionalFourierConvolutions}. The sharp exceptional set estimate in the plane was recently obtained by Ren--Wang \cite{RW23} but in higher dimensions less is known.  We are able to improve the state of the art in higher dimensions under only very mild assumptions on the set itself, see Corollaries~\ref{genericimprove}, \ref{genericimprove2}, and Section~\ref{sec:example} where we consider an illustrative example.

\subsection{Notation}
Throughout this paper we write $A\lesssim B$ if there exists a constant $C>0$ such that $A\leq CB$ and $A\approx B$ if $A\lesssim B$ and $B\lesssim A$. If we wish to emphasise that the constant $C$ depends on some parameter $\lambda$ we shall express it as $A\lesssim_{\lambda}B$ or $A\approx_{\lambda}B$.

For a set $X\subseteq\rd$, we shall write $\M(X)$ for the set of Borel measures $\mu$ compactly supported on $X$ and such that $0<\mu(X)<\infty$.

For integers $1\leq k<d$, we write $G(d,k)$ for the Grassmannian manifold of $k$-dimensional planes $V\subseteq\rd$ equipped with the invariant Borel probability measure $\gamdk$ obtained from the Haar measure on the topological group of rotations around the origin. We write  $P_{V}:\rd\to V$ for the orthogonal projection onto $V\in\gdk$, which we identify with $\rk$. Given a measure $\mu\in\M(\rd)$, we will write $\mu_{V}$ for the push-forward measure under $P_{V}$, i.e. $\mu_{V}(B) = \mu\big( P_{V}^{-1}(B) \big)$ for Borel sets $B\subseteq\rk$. Given $y \in \mathbb{R}^k$ and $V\in\gdk$, we define  $y_{V} \in \rd$ by
\begin{equation}\label{eq:yv}
\{y_{V} \} = V\cap P_{V}^{-1}(y).
\end{equation}
We write $\hd$ to denote the Hausdorff dimension of a set or Borel measure and $\mathcal{H}^s$ to denote the $s$-dimensional Hausdorff measure.

\section{Preliminaries}
 
\subsection{Hausdorff dimension and orthogonal projections}\label{section:preliminariesProj}
One of the most well-known and influential results in Fractal Geometry is Marstrand's projection theorem.  This was proved in the plane by Marstrand   \cite{Mar54} and in higher dimensions by Mattila \cite{Mat75} and states that for Borel sets $X\subseteq\rd$ and $\gamdk$ almost all $V\in\gdk$, $\hd P_{V}(X) = \min\{k, \hd X \}$.

Marstrand's projection theorem motivated the study of the dimension of the set of projections for which the conclusion  does  not hold: the set of exceptional projections. Kaufman \cite{Kau68} addressed this question in the context of compact sets $X\subseteq\R^2$, and proved that for $u\in[0, \hd X]$,
\begin{equation}\label{eq:kaufmanbound}
		\hd \{ V\in\gto : \hd P_{V}(X)< u \}\leq u,
\end{equation}
which was proven to be sharp when $u=\hd X$ by Kaufman and Mattila in \cite{KM75}.
Kaufman's proof of \eqref{eq:kaufmanbound} was generalised in \cite[Theorem~5.1]{Mat15} to show that if $X\subseteq\rd$ has $\hd X\leq1$ then for all $u\in[0,\hd X]$,
\begin{equation}\label{eq:kaufmangeneral}
    \hd\{ V\in\gdo : \hd P_{V}(X)<u \}\leq d-2+u.
\end{equation}
Later, using different methods, Bourgain \cite[Theorem~4]{Bou10} and Oberlin \cite[Theorem~1.2]{Obe12} proved that for Borel sets $X\subseteq\rr$,
\begin{equation}\label{eq:Oberlinbound}
	\hd\{ V\in\gto : \hd P_{V}(X)< \hd X/2 \} = 0,
\end{equation}
and Oberlin conjectured that if $ \frac{\hd X}{2} \leq u \leq \min\{\hd X, 1\}$, then
\begin{equation}\label{eq:oberlin}
		\hd \{ V\in\gto : \hd P_{V}(X)<u \}\leq 2u-\hd X.
\end{equation}
Oberlin's conjecture \eqref{eq:oberlin} was recently proved in a breakthrough paper of Ren and Wang  \cite[Theorem~1.2]{RW23}, which built on significant recent progress  from various people; see, for example,  \cite{GSW19,OS23,OS23+,OSW24} and the references therein.  In fact, \cite[Theorem~4]{Bou10} provides a stronger `$\varepsilon$-improvement' on \eqref{eq:Oberlinbound}, and this  played an important role in the development of the problem, including in \cite{RW23}.

Most of these results have analogues to higher-dimensional projections onto planes $V\in\gdk$ for $k\geq2$. For example, \cite[Corollary~5.12]{Mat15} generalised \eqref{eq:kaufmangeneral} to show that if $\hd X \leq k$, then for any $u\in[0,\hd X]$
\begin{equation}\label{eq:mattilaboundgeneral}
    \hd\{ V\in \gdk : \hd P_{V}(X)<u \}\leq k(d-k-1)+u,
\end{equation}
which was shown in \cite[Theorem~5]{KM75} to be sharp when $u = \hd X\leq k$.
Peres and Schlag \cite[Proposition 6.1]{PS00} gave the following upper bound for Borel sets $X\subseteq\rd$ without restrictions on their Hausdorff dimension. For $u\in[0,k]$,
\begin{equation}\label{eq:peresschlagbound}
    \hd\{ V\in \gdk : \hd P_{V}(X)< u \}\leq k(d-k)+u-\hd X.
\end{equation}
This generalised the result of Falconer \cite[Theorem~1~(i)]{Fal82} for $u=1$ and is sharp for $u=k$ as shown in \cite[Example 5.13]{Mat15}. Note that Peres and Schlag's bound is only better than \eqref{eq:mattilaboundgeneral} when $\hd X>k$. Although their bound is stated only for measures, for the same reason that we give in the proof of Theorem~\ref{thm:exceptionalFourier}, the result holds for sets.

Bourgain's estimate \eqref{eq:Oberlinbound} was generalised by He in \cite{He20}, which gives as a corollary \cite[Corollary~3]{He20} that for Borel sets (or more generally, analytic) $X\subseteq \rd$ with $\hd X<d$,
\begin{equation*}
    \hd\{ V\in \gdk : \hd P_{V}(X)< k\hd X/d\}\leq k(d-k)-1.
\end{equation*}

\subsection{Dimension and Fourier transforms}

Recall that Frostman's lemma allows us to write the Hausdorff dimension of a Borel set $X$ in terms of the $s$-energy $I_{s}$ of measures $\mu\in\M(X)$,
\begin{equation*}
		I_{s}(\mu) = \iint |x-y|^{-s}\,d\mu(x)\,d\mu(y),
\end{equation*}
as
\begin{equation*}
		\hd X = \sup\{ s\geq0 : \exists \mu\in\M(X) : I_{s}(\mu)<\infty \}.
\end{equation*}

When using this definition of the Hausdorff dimension, techniques are usually referred to as potential-theoretic, see \cite{Fal03} for more about Hausdorff dimension and energy integrals. In fact, \cite{Kau68} gave an elegant  potential-theoretic proof of Marstrand's theorem using arguments from Fourier analysis.

Fourier transforms can also be used to represent energy integrals. If $\mu$ is a finite Borel measure, we define its \emph{Fourier transform} by
\begin{equation*}
		\widehat{\mu}(z) = \int e^{-2\pi iz\cdot x}\,d\mu(x).
\end{equation*}
Using that the Fourier transform of $|z|^{-s}$ is, in the distributional sense, a constant multiple of $|z|^{s-d}$, and Parseval's theorem, the $s$-energy of $\mu\in\M(\rd)$, for $s\in(0,d)$, may be expressed as
\begin{equation*}
		I_{s}(\mu) \approx_{d,s} \int \big| \widehat{\mu}(z) \big|^2 |z|^{s-d}\,dz.
\end{equation*}

Note that the $s$-energy can be interpreted as a Sobolev norm. Thus, having finite energy gives information regarding the smoothness of the measure. This relation between the Hausdorff dimension of sets and the Fourier transform of measures they support motivates the definition of \emph{Fourier dimension} of measures
\begin{equation*}
		\fd \mu = \sup\big\{ s\geq0 : \big| \widehat{\mu}(z) \big|\lesssim |z|^{-\frac{s}{2}} \big\},
\end{equation*}
and of Borel sets $X\subseteq\rd$
\begin{equation*}
		\fd X = \sup\big\{ \min\{ \fd \mu,  d\} : \mu\in\M(X) \big\}.
\end{equation*}

For a Borel set $X\subseteq\rd$, $0\leq\fd X\leq\hd X\leq d$ and all of these inequalities can be strict in any combination. Sets for which the Fourier and Hausdorff dimensions coincide are called \emph{Salem sets}. Constructing non-trivial deterministic Salem sets is challenging, but random examples abound. The first construction of a Salem set was a random Cantor set in $\R$ of any dimension between $0$ and $1$ given by Salem \cite{Sal51}. Kaufman \cite{Kau81} calculated the Fourier dimension of a set defined by Jarnik in \cite{Jar31}, thus giving the first explicit construction of a Salem set. Some examples of explicit Salem sets in $\rd$ for $d\geq2$ can be found in \cite{FH23} and \cite{Gat67}. We refer the reader to \cite{Ham17} for a more detailed summary of the history of Salem sets. 

An important (and simple) observation is that for Salem sets, there are no exceptional directions for the projections, see the beginning of Section~\ref{section:FDproj}. In fact, one can say slightly more.  If $X\subseteq\rd$ is a Borel set with $\fd X = t$, then for all $u\leq t$,
\begin{equation*}
		\big\{ V\in\gdk : \hd P_{V}(X) < \min\{k,u\} \big\} =\varnothing.
\end{equation*}
However, we show in Proposition~\ref{propo:counter_example} that nothing can be said for $u>t$ using the Fourier dimension alone.

The definition of Hausdorff dimension for sets using energy integrals can be extended naturally to measures, defining the \emph{Sobolev dimension} of a measure $\mu\in\M(\rd)$ by
\begin{equation*}
	\sd \mu = \sup\bigg\{ s \in \R : \int_{\rd}\big|\widehat{\mu}(x)\big|^2 |x|^{s-d}\,dx<\infty \bigg\}.
\end{equation*}
This concept goes back to Peres--Schlag \cite{PS00}; see also \cite{Mat15}.  For any Borel measure $\mu \in \M(\rd)$, $0 \leq \fd \mu\leq\sd \mu$ and $\hd \mu \geq \min\{d, \sd \mu\}$.  Contrary to what one might expect, both the Fourier and Sobolev dimensions of measures may exceed the Hausdorff dimension of the ambient space. Take as an example the Lebesgue measure restricted to $[0,1]$, $\L^1\big|_{[0,1]}$. This measure satisfies $\fd \L^1\big|_{[0,1]} = \sd\L^1\big|_{[0,1]} = 2>\hd[0,1]$.

\subsection{Fourier spectrum}

In \cite{Fra22}, the first-named author defined the Fourier spectrum, a family of dimensions lying between the Fourier and the Sobolev dimension of measures. For this, he defined the $(s,\theta)$-energies of a measure $\mu\in\M(\rd)$, for $\theta\in(0,1]$ and $s\geq0$, as
\begin{equation*}
	\J_{s,\theta}(\mu) = \bigg( \int_{\rd} \big| \widehat{\mu}(z) \big|^{\frac{2}{\theta}}|z|^{\frac{s}{\theta}-d}\,dz \bigg)^\theta,
\end{equation*}
and for $\theta = 0$,
\begin{equation*}
	\J_{s,0}(\mu) = \sup_{z\in\rd} \big| \widehat{\mu}(z) \big|^2|z|^{s}.
\end{equation*}
Then the \emph{Fourier spectrum} of $\mu$ at $\theta$ is
\begin{equation*}
	\fs \mu = \sup\{ s \in \R : \J_{s,\theta}(\mu)<\infty \},
\end{equation*}
and for each $\theta\in[0,1]$, $\fd \mu\leq\fs \mu\leq\sd \mu$, with equality on the left if $\theta= 0$ and equality on the right if $\theta = 1$. As a function of $\theta$, $\fs\mu$ is continuous for $\theta\in(0,1]$ by \cite[Theorem~1.1]{Fra22} and, in addition,  continuous at $\theta=0$ provided $\mu$ is compactly supported by \cite[Theorem~1.3]{Fra22}.

For  a Borel set $X\subseteq\rd$, the \emph{Fourier spectrum} is
\begin{equation*}
	\fs  X = \sup\big\{ \min\{\fs \mu, d\} : \mu\in\M(X) \big\}.
\end{equation*}
Then,  for all $\theta\in[0,1]$, $\fd X\leq\fs  X \leq \hd X$, with equality on the left if $\theta= 0$ and equality on the right if $\theta = 1$. Moreover,  $\fs X$ is continuous for all $\theta\in[0,1]$ by \cite[Theorem~1.5]{Fra22}.

\section{Fourier dimension and exceptional sets for the Hausdorff dimension}

We  begin by building an example to show the pointwise sharpness of the inequality \eqref{eq:oberlin} proven in \cite[Theorem~1.2]{RW23}.   The  construction, which was described to us by Tuomas Orponen, is similar to the one given in \cite{KM75}, however, more care is required since there are more parameters to look after.

\begin{lma}\label{lemma:sharp1}
	For all $s\in(0,1]$ and all $u\in(0,s)$ there exists a compact set $X_{u,s}\subseteq\R^2$ with $\hd X_{u,s} = s$ such that
	\begin{equation*}
		\hd \{ V\in\gto : \hd P_{V}(X_{u,s})\leq u \} = \max\{ 0, 2u-s\}.
	\end{equation*}
\end{lma}
\begin{proof}
	 Let $s\in(0,1]$. If $u<s/2$ then by \eqref{eq:Oberlinbound} we know that the dimension of the exceptional set is $0$. Thus, we only consider the case $u\in[s/2,s)$. Let $(\eta_m)_{m\in\N}$ be a rapidly increasing sequence of positive integers; say $\eta_{m+1} \geq \eta_{m}^m$, and define the sets
\begin{align*}
		A &= \big\{ x\in[0,1] : \forall m\in\N, d(x,\eta_{m}^{-u}\Z)\leq \eta_{m}^{-1} \big\};\\
		B &= \big\{ x\in[0,1] : \forall m\in\N, d(x,\eta_{m}^{-(s-u)}\Z)\leq \eta_{m}^{-1} \big\};\\
		C &= \big\{ x\in[0,1] : \forall m\in\N, d(x,\eta_{m}^{-(2u-s)}\Z)\leq \eta_{m}^{-1} \big\}.
\end{align*}
Here  $d(x,Y) = \inf\{ |x-y| :  y \in Y\}$. By \cite[Theorem 10]{Egg52} these sets have $\hd A = u,~\hd B =s-u$ and $\hd C=2u-s$, and, since their Hausdorff and packing dimensions are equal, 
\begin{equation*}
		\hd (A\times B) = \hd A + \hd B = s.
\end{equation*}
By Marstrand's theorem, since $s \leq 1$, then for $\gamma_{2,1}$ almost all $V\in\gto$, $\hd P_{V}(A\times B) = s$.  Also, $\hd (A+cB) = 2u-s$ for all $c\in C$. To prove this, fix $m\in\N$ and define the discrete sets
\begin{align*}
	A_{m} &= \{ \eta_{m}^{-u}z : z\in\Z \};\\
	B_{m} &= \{ \eta_{m}^{-(s-u)}z : z\in\Z \};\\
	C_{m} &= \{ \eta_{m}^{-(2u-s)}z : z\in\Z \}.
\end{align*}
The projection of a point of the form $(\eta_{m}^{-u}z_{1},\eta_{m}^{-(s-u)}z_{2})\in A_{m}\times B_{m}$ onto a line with slope $\eta_{m}^{-(2u-s)}z_{3}\in C_{m}$ is, up to a constant scaling factor, of the form
\begin{equation*}
	(\eta_{m}^{-u}z_{1},\eta_{m}^{-(s-u)}z_{2})\cdot (1,\eta_{m}^{-(2u-s)}z_{3}) = \eta_{m}^{-u}z_1 + \eta_{m}^{-u}(z_{2}z_{3})  = \eta_{m}^{-u}(z_1 + z_{2}z_{3}).
\end{equation*}
This implies that, for each $m\in\N$, the projections of $A_{m}\times B_{m}$ onto lines with slopes given by $C_{m}$ are contained in $A_{m}$.

Now, for each $\eta_{m}^{-(2u-s)}z_{3}\in C_{m}$, consider the interval of length $2\eta_{m}^{-1}$ around $\eta_{m}^{-(2u-s)}z_{3}$.  The union of these intervals constitutes a set of directions for which the projection of $A \times B$ has dimension at most $\hd A=u$ at scale $\eta_m^{-1}$ (that is, the projection may be covered by a constant times $\eta_m^{u}$ many intervals of length $\eta_m^{-1}$). Intersecting these sets  over all $m\in\N$ we get the set $C$. This shows that for the directions with slopes given by $C$, the Hausdorff dimension of the projection of $A\times B$ is at most $\hd A = u$. It then follows that the exceptional set
\begin{equation*}
		\{ V\in\gto : \hd P_{V}(A\times B) \leq u \},
\end{equation*}
contains a copy of $C$. Thus,
\begin{align*}
		\hd \{ V\in\gto : \hd P_{V}(A\times B) \leq u \} &\geq\hd C \\
&=2u-s\\
&=2u-\hd (A \times B).
\end{align*}
Since the reverse inequality is true by \cite[Theorem~1.2]{RW23}, letting $X_{u,s} = A\times B$ we have the desired result.
\end{proof}

\subsection{Fourier dimension and orthogonal projections}\label{section:FDproj}

Given a measure $\mu\in\M(\rd)$, an integer $1\leq k<d$ and $V\in\gdk$,
\begin{equation}\label{eq:projmeasure}
    \widehat{\mu_{V}}(y) = \int_{\rd}e^{-2\pi iy\cdot P_{V}(x)}\,d\mu (x)= \int_{\rd}e^{-2\pi i y_V\cdot x}\,d\mu (x) = \widehat{\mu}(y_{V})
\end{equation}
recalling the definition of $y_{V}$ from \eqref{eq:yv}. Thus, for all $V\in\gdk$ and finite Borel measures, $\fd\mu\leq\fd\mu_{V}$ which implies that if $X\subseteq\rd$ is a Borel set, then for all $V\in\gdk$, $\hd P_{V}(X) \geq \fd P_{V}(X) \geq \min\{ k,\fd X \}$. In particular, for Salem sets, there are no exceptions to Marstrand's theorem. We show in the following proposition that one cannot say anything more than this using the Fourier dimension alone.  That is, knowledge of the Fourier dimension of $X$ does not give any information about the dimension of the set of $V$ for which $\hd P_{V}(X) < u$ as soon as $u> \fd X$.

\begin{prop}\label{propo:counter_example}
	For any $s\in(0,1]$ and $t\in(s/2,s)$ there exists a compact set $X\subseteq\R^2$ with $\hd X = s$ and $\fd X = t$ such that for $u<t$,
\begin{equation*}
    \hd\{ V\in\gto : \hd P_{V}(X)\leq u \} =0,
\end{equation*}
and for $u\geq t$,
\begin{equation*}
  \hd\{ V\in\gto : \hd P_{V}(X)\leq u \} \geq 2t-s.
\end{equation*}
That is, the dimension of the exceptional set has a jump discontinuity at $\fd X$ from $0$ to the largest value it could possibly take according to \eqref{eq:oberlin}.
\end{prop}
\begin{proof}
	Let $s\in(0,1]$, $t\in(s/2,s)$ and $A\coloneqq X_{t,s}\subseteq\R^2$ be as in Lemma~\ref{lemma:sharp1}, that is, $A$ is compact,  $\hd A = s$, and
\begin{equation} \label{dimbound1}
		\hd \{ V\in\gto : \hd P_{V}(A)\leq t \} = 2t-s.
\end{equation}
Using \cite{FH23}, define a compact Salem set $B\subseteq\R^2$ of dimension $t$ and assume that $A$ and $B$ are disjoint. Clearly $\fd A \leq  t$ (otherwise \eqref{dimbound1} could not hold) and so  $\fd (A\cup B) = t$, and for each $u< t$,
\begin{equation*}
	\hd \{ V\in\gto : \hd P_{V}(A\cup B) \leq u \} =0.
\end{equation*}
Also,
\begin{equation*}
\hd P_{V}(A\cup B)= \max\big\{ \hd P_{V}(A), \hd P_{V}(B) \big\},
\end{equation*}
and by the discussion at the beginning of Section~\ref{section:FDproj}, $\hd P_{V}(B) = t$ for all $V\in\gto$. Then, 
\begin{equation*}
	\hd \{ V\in\gto : \hd P_{V}(A\cup B) \leq t \} = \hd \{ V\in\gto : \hd P_{V}(A)\leq t \} = 2t-s
\end{equation*}
and therefore, if $u\geq t$,
\begin{equation*}
      \hd \{ V\in\gto : \hd P_{V}(A\cup B) \leq u \} \geq 2t-s.
\end{equation*}
Letting $X = A\cup B$ gives the result.
\end{proof}

This last proposition showed us that the Fourier dimension alone is perhaps  not strong enough to capture the effect that the decay of the Fourier transform of measures has on  projections. We will see later that the Fourier spectrum can say more, see Proposition~\ref{prop:improveOberlin}, Theorem~\ref{thm:exceptionalFourier} and Corollary~\ref{thm:exceptionalFouriercoro}.

\section{Projection theorems for the Fourier spectrum}\label{section:FourierProj}

Marstrand type theorems for other dimensions are well studied.  For example, there are Marstrand type theorems for the box and packing dimensions, where the dimension of the projection is almost surely constant. For a more detailed discussion of such results we refer the reader to \cite{FFJ15, Mat14} and the references therein. On  the other hand,  in the case of the Assouad dimension, a surprising result holds: the Assouad dimension of orthogonal projections need not be almost surely  constant \cite{FO17}.  

It is natural to try to answer these types of questions for the Fourier spectrum. A useful lower bound can be obtained for almost all directions   following Kaufman's potential-theoretic proof of Marstrand's theorem, \cite{Kau68}.
\begin{thm}\label{thm:spectrumProj}
  Let $\mu\in\M(\rd)$, $X\subseteq\rd$ be a Borel set and $1\leq k<d$ be an integer. Then, for $\gamdk$ almost all $V\in G(d,k)$, for all  $\theta\in(0,1]$, $\fs\mu_{V}\geq \fs\mu$ and $\fs P_{V}(X)\geq\min\{ k,\fs X \}$.
\end{thm}
\begin{proof}
First fix $\theta\in(0,1]$.  Integrating the energies of the projected measure with respect to planes in $\gdk$, recalling \eqref{eq:yv} and \eqref{eq:projmeasure}, and applying Fubini's theorem,   gives
  \begin{align*}
    \int_{\gdk}\J_{s,\theta}(\mu_{V})^{1/\theta}\,\ddk &= \int_{\gdk}\int_{\rk}\big| \widehat{\mu_{V}}(y) \big|^{\frac{2}{\theta}}|y|^{\frac{s}{\theta}-k}\, dy \,\ddk\\
    &= \int_{\rk}\int_{\gdk}\big| \widehat{\mu}(y_{V}) \big|^{\frac{2}{\theta}}|y|^{\frac{s}{\theta}-k}\,\ddk\, dy .
  \end{align*}
  Note that $\rd = \R^{d-k+1}\times\R^{k-1} = ( \R S^{d-k})\times\R^{k-1}$, where $S^{d-k}$ denotes the sphere in $\mathbb{R}^{d-k+1}$ and $\R^{d-k+1} = \R S^{d-k} = \{ re: r \in \mathbb{R}, \ e \in S^{d-k}\}$ is the usual representation of $\R^{d-k+1}$ in spherical coordinates.  Fix $y\in\rk$ and let $\pi_{y} : G(d,k)\to S^{d-k}$ be defined by  $\pi_{y}(V) = e$ where $e\in S^{d-k}$ is chosen such that $y_{V} = y_{W_e}$ for $W_e : = \{ (re,v) : r\in\R, v\in\R^{k-1} \} \in G(d,k)$. Note that, for almost all $y\in\rk$, the choice of $e$ is unique for all $V \in G(d,k)$ and so  $\pi_{y}$ is a well-defined  Borel measurable function almost surely.  By $O(d-k)$ rotational symmetry, the push-forward of $\gamdk$ by $\pi_y$ (for almost all $y$)  is simply  $\sigma^{d-k}$; the surface measure on the sphere $S^{d-k}$.  Therefore, for almost all $y\in\rk$, by the disintegration theorem for measures,  for each $e\in S^{d-k}$ there exists a $\sigma^{d-k}$ almost everywhere uniquely determined family of Borel probability measures $\{ \gamma_{y,e} \}_{e\in S^{d-k}}$ on $\gdk$ such that
  \begin{equation*}
      \int_{\gdk}\big| \widehat{\mu}(y_{V}) \big|^{\frac{2}{\theta}}|y|^{\frac{s}{\theta}-k}\,\ddk = \int_{S^{d-k}}\int_{\pi_{y}^{-1}(e)}\big| \widehat{\mu}(y_{V}) \big|^{\frac{2}{\theta}}|y|^{\frac{s}{\theta}-k}\,d\gamma_{y,e}(V)\,d\sigma^{d-k}(e).
  \end{equation*}
 Therefore, since   $y_{V} = y_{W_e}$ only depends on $e$,
  \begin{equation*}
    \int_{\gdk}\big| \widehat{\mu}(y_{V}) \big|^{\frac{2}{\theta}}|y|^{\frac{s}{\theta}-k}\,\ddk = \int_{S^{d-k}}\big| \widehat{\mu}(y_{W_e}) \big|^{\frac{2}{\theta}}|y|^{\frac{s}{\theta}-k}\,d\sigma^{d-k}(e).
  \end{equation*}
  Thus, disintegrating  Lebesgue measure on   $\rd = ( \R S^{d-k})\times\R^{k-1} =  \cup_{ e \in S^{d-k}} W_e$  using $\sigma^{d-k}$ on $S^{d-k}$  and $\H^k$ on $W_e$ as $dz \approx |y|^{d-k}  \, \dhky \,d\sigma^{d-k}(e)$, we get
  \begin{align*}
    \int_{\gdk}\J_{s,\theta}(\mu_{V})^{1/\theta}\,\ddk &= \int_{\rk}\int_{S^{d-k}}\big| \widehat{\mu}(y_{W_e}) \big|^{\frac{2}{\theta}}|y|^{\frac{s}{\theta}-k}\,d\sigma^{d-k}(e)\, dy \\
 &\approx \int_{S^{d-k}}\int_{W_e}\big| \widehat{\mu}(y) \big|^{\frac{2}{\theta}}|y|^{\frac{s}{\theta}-k}\, \dhky \,d\sigma^{d-k}(e)\\
    &\approx \int_{\rd}\big| \widehat{\mu}(z) \big|^{\frac{2}{\theta}}|z|^{\frac{s}{\theta}-d}  \,dz\\
    &= \J_{s,\theta}(\mu)^{1/\theta}.
  \end{align*}
  Then, if $s<\fs\mu$, $\J_{s,\theta}(\mu)<\infty$ and for $\gamdk$ almost all $V\in\gdk$, $\J_{s,\theta}(\mu_{V})<\infty$.   Therefore, $ \fs\mu_{V} \geq s$ for $\gamdk$ almost all $V\in\gdk$.  This proves  the statement for measures holds \emph{pointwise}, that is, for all $\theta\in(0,1]$ it holds for $\gamdk$ almost all $V\in G(d,k)$.  However, since the Fourier spectrum of a compactly supported measure  is continuous, and thus determined on a countable set, we can immediately upgrade this pointwise result such that  for $\gamdk$ almost all $V\in G(d,k)$ it holds  \emph{simultaneously} for all $\theta$, as required. 
  
  The claim for sets follows immediately from the claim for measures.
\end{proof}

Recall that when $\theta=0$ the result of the previous theorem holds for \emph{all} $V\in\gdk$, and when $\theta=1$ the reverse inequality for sets is  valid for \emph{all} $V\in\gdk$ due to the fact that projections are Lipschitz maps.

The proof of Theorem \ref{thm:spectrumProj} is rather simpler in the case $k=1$ since then $G(d,k)$ may be identified with $S^{d-1}$ and $\rd$ thus expressed in spherical coordinates as $\rd =   \mathbb{R}G(d,1)$.  The complication for $k \geq 2$ arises because parametrising $\rd$ by points in $G(d,k) \times \mathbb{R}^k$ leads to `redundancy' which must then be integrated out.  Setting $\theta = 1$ and recalling the simple fact that the Hausdorff dimension cannot increase under orthogonal projections yields Marstrand's projection theorem for sets. To the best of our knowledge this is a new proof of Marstrand's theorem in the case $k \geq 2$; see \cite[proof of Theorem 4.1]{Mat15} for the $\theta=k=1$ case, treated in the same (but simpler)  way.

From \cite[Proposition~4.2]{CFdO24} we know that $\fs\mu\leq\fd\mu + d\theta$. Combining this with the previous argument gives for all $V\in\gdk$ and $\theta\in[0,1]$,
\begin{equation*}
	\fs \mu \leq\fd \mu + d\theta\leq\fd \mu_{V} + d\theta \leq\fs \mu_{V} + d\theta.
\end{equation*}
The following proposition shows that in fact, a stronger inequality holds.

\begin{prop}\label{prop:boundmuV}
  Let $\mu\in\M(\rd)$, $X\subseteq\rd$ be a Borel set, $\theta\in[0,1]$ and $1\leq k<d$ an integer. For all $V\in G(d,k)$, $\fs\mu_{V} \geq \fs\mu- (d-k)\theta$ and $\fs P_{V}(X) \geq \min\{k, \fs X - (d-k)\theta\}$.
\end{prop}
\begin{proof}
  If $\theta =0$ the result is trivial by the discussion at the beginning of Section~\ref{section:FDproj}. Let $R$ be larger than the diameter of the support of both $\mu$ and $\mu_{V}$ and $0<\alpha<1/R$. By \cite[Theorem~3.1]{CFdO24}, we have for any $s\geq0$ and $\theta\in(0,1]$,
  \begin{equation*}
    \J_{s,\theta}(\mu)^{1/\theta} \approx 1+ \sum_{z\in \alpha\zd \setminus\{0\}}\big| \widehat{\mu}(z) \big|^{\frac{2}{\theta}}|z|^{\frac{s}{\theta}-d}.
  \end{equation*}
    Therefore, identifying $V\in G(d,k)$ with $\rk$ and applying \cite[Theorem~3.1]{CFdO24} again,
  \begin{equation*}
      \J_{s,\theta}(\mu)^{1/\theta}\gtrsim 1+  \sum_{z\in \alpha\Z^k\setminus\{0\}} \big| \widehat{\mu_{V}}(z) \big|^\frac{2}{\theta}|z|^{\frac{s-(d-k)\theta}{\theta}-k}\approx \J_{s-(d-k)\theta,\theta}(\mu_{V})^{1/\theta},
  \end{equation*}
  which proves the result.

  The claim for sets follows immediately from the claim for measures.
\end{proof}

As an immediate consequence of the previous  proposition, we get some new non-trivial information about the exceptional set for Hausdorff dimension. 
\begin{prop}\label{prop:improveOberlin}
  Let $X\subseteq\rd$ be Borel and $1\leq k< d$ an integer. If $u\leq\sup_{\theta\in[0,1]}\big(\fs X - (d-k)\theta\big)$, then
    \begin{equation*}
        \{ V\in G(d,k) : \hd P_{V}(X)<u \} = \varnothing.
    \end{equation*}
\end{prop}
\begin{proof}
  Since $\fs P_{V}(X)\leq\hd P_{V}(X)$ for all $\theta\in[0,1]$, then
  \begin{equation*}
      \{ V\in G(d,k) : \hd P_{V}(X)<u \}\subseteq \{ V\in G(d,k) : \fs P_{V}(X)<u \}.
  \end{equation*}
  Thus, the result follows  from Proposition~\ref{prop:boundmuV}.
\end{proof}

With this, if $X\subseteq\rr$ is a Borel set such that for some $\theta\in[0,1]$, $\theta<\fs X$, we get an improvement on the  Bourgain--Oberlin  bound \eqref{eq:Oberlinbound} in the sense that the exceptional set is empty, not just of Hausdorff dimension $0$, once $u$ is small enough.  Furthermore, if for some $\theta\in[0,1]$,  $\tfrac{\hd X}{2}+\theta<\fs X$, then we get an improved range in the Bourgain--Oberlin bound (in addition to the upgrade from dimension 0 to empty).  Given that for Borel sets in $\rr$, $\fs X \leq \min\{\fd X+2 \theta, \hd X\}$ by  \cite[Proposition~4.2]{CFdO24}, this latter improvement is only possible for $\tfrac{\hd X}{2} -\fd X < \theta < \tfrac{\hd X}{2}$.  In particular, we must have $\fd X >0$ but, if $\fd X < \tfrac{\hd X}{2}$, then the improvement will not come from the Fourier dimension directly and can only be achieved by the Fourier spectrum.

\begin{ques}
  Is there a Marstrand theorem for the Fourier spectrum?  More precisely, fix a Borel set or finite Borel measure in $\rd$, an integer $1 \leq k < d$ and $\theta \in [0,1]$: is it true that  the value of the Fourier spectrum at $\theta$ of the projection onto $V$ is the same for almost all $V \in G(d,k)$? 
\end{ques}

An intriguing special case of this question concerns the Fourier dimension (that is, when $\theta=0$) and we are unaware of any progress on this front.  On the other hand, for $\theta=1$ we know the answer to the question  is yes for sets  by Marstrand's theorem and for measures $\mu\in\M(\rd)$ with $\sd\mu\leq k$ by \cite[Theorem~1.1]{HK97}.  We note that if this `pointwise' question could be answered in the affirmative then, using continuity of the Fourier spectrum, the result could be upgraded to hold almost surely, for all $\theta \in [0,1]$ simultaneously.

\section{Exceptional set estimates for the Fourier spectrum}\label{section:exceptionalSpectrum}

Recall Proposition~\ref{propo:counter_example}, in which we showed that the Fourier dimension does not help improve the bound of the exceptional set of projections for values $u>\fd X$. The following theorem shows that the Fourier spectrum can do better. When $\theta = 1$ we recover the bound from \cite[Proposition~6.1]{PS00}, which also inspired our proof.

\begin{thm}\label{thm:exceptionalFourier}
	Let $\mu\in\M(\rd)$, $\theta\in(0,1]$ and $1\leq k<d$ be an integer. Then for all $u \geq 0$,
 \begin{equation}\label{eq:thm1}
 	 \hd \{ V\in G(d,k) : \fs \mu_{V}<u \}\leq  \max\bigg\{ 0, k(d-k)+\frac{u-\fs \mu}{\theta}\bigg\}.
 \end{equation}
 Furthermore, if $X$ is a Borel set in $\rd$ and $\theta\in(0,1]$, then for all $u\in[0,k]$,
 \begin{equation*}
 	\hd \{ V\in G(d,k) : \fs  P_{V}(X)<u \}\leq  \max\bigg\{ 0, k(d-k)+\frac{u-\fs X}{\theta}\bigg\}
 \end{equation*}
\end{thm}
\begin{proof}
    The claim for sets follows from the statement for measures. To see this fix $\theta\in(0,1]$, let $\varepsilon>0$ and $\mu\in\M(X)$ be such that $\fs\mu \geq  \fs X - \varepsilon$. Since $\fs P_{V}(X) \geq \min\{ k, \fs \mu_{V}\}$ and $u \leq k$, it follows that $\fs P_{V}(X)<u \Rightarrow  \fs\mu_{V}<u$.  Then,
  \begin{align*}\label{eq:suff}
    \hd\{ V\in\gdk : \fs P_{V}(X)<u \}&\leq\hd\{ V\in\gdk : \fs\mu_{V}<u \}\\
    &\leq k(d-k)+\frac{u-\fs\mu}{\theta}\\
    &= k(d-k)+\frac{u-\fs X+\varepsilon}{\theta},
  \end{align*}
 and  letting $\varepsilon\to0$ gives the result. We now proceed to prove the claim for measures, which follows the general strategy of  \cite{PS00}; see also \cite{Mat15}.
   
  Let $G_{u,\theta} = \{ V\in\gdk : \fs  \mu_{V}<u \}$ and suppose \eqref{eq:thm1} is false for some $u>0$. Choose $\tau>0$ such that $k(d-k)+\frac{u-s}{\theta}<\tau<\hd G_{u,\theta}$, for some $s<\fs \mu$.  First, observe that $G_{u,\theta}$ is a Borel set.  Indeed, 
\begin{align*}
  \gdk \setminus G_{u,\theta} &= \bigcap_{\varepsilon \in (0,1) \cap \mathbb{Q}} \bigcup_{n \in \mathbb{N}} \bigcap_{m \in \mathbb{N}} \left\{V \in \gdk :  \int_{|y| < m} \big|\widehat{\mu_V}(y)\big|^{\frac{2}{\theta}} |y|^{\frac{u-\varepsilon}{\theta}-k} \, dy < n\right\}\\
  &\eqqcolon  \bigcap_{\varepsilon \in (0,1) \cap \mathbb{Q}} \bigcup_{n \in \mathbb{N}} \bigcap_{m \in \mathbb{N}} A_{\varepsilon, n, m}
\end{align*}
and since, recalling  \eqref{eq:projmeasure}, $V \mapsto \widehat{\mu_V}(y) = \widehat{\mu}(y_{V})$ is continuous (since the Fourier transform of $\mu$ is continuous and $y_V$ depends continuously on $V$) the set $ A_{\varepsilon, n, m}$ is open and thus $G_{u,\theta}$ is Borel. Therefore, by Frostman's lemma there exists a measure $\nu\in\M(G_{u,\theta})$ such that $\nu\big( B(V,r) \big)\leq r^\tau$ for all $V\in\gdk$ and $r>0$. We will show that
	\begin{equation}\label{eq:thm1_energy}
		\int_{\gdk}\J_{u,\theta}(\mu_{V})^{1/\theta}\,d\nu(V)< \infty,
	\end{equation}
	and this will imply that $\J_{u,\theta}(\mu_{V})^{1/\theta}<\infty$ for $\nu$ almost every $V\in \gdk$. Then $\nu(G_{u,\theta})=0$ which contradicts the fact that $\nu\in\M(G_{u,\theta})$.
	
We write $\S(\rd)$ for the family of  functions in the Schwartz class on $\rd$, see  \cite[Chapter 3]{Mat15}.	Let $\varphi\in\S(\rd)$, such that $\varphi(x) = 1$ in $\spt\mu$, where $\spt\mu$ denotes the (compact) support of $\mu$.  Then $\mu = \varphi \mu$ and $\widehat{\mu} = \widehat{\varphi \mu} = \widehat{\varphi}*\widehat{\mu}$. Moreover,  $\widehat{\varphi}\in\S(\rd)$, and for every $N\in\N$, $\big|\widehat{\varphi}(z)\big|\lesssim_{\varphi,N}\big(1+|z|\big)^{-N}$. Using H\"older's inequality with conjugate exponents $2/\theta$ and $2/(2-\theta)$, we obtain the following estimate for $z \in \rd$.
	\begin{align*}
		\big|\widehat{\mu}(z)\big|^{\frac{2}{\theta}} &\leq \bigg[ \int_{\rd}\big|\widehat{\mu}(z-y)\widehat{\varphi}(y)\big|\,dy \bigg]^{\frac{2}{\theta}}\\
		&= \bigg[ \int_{\rd} \big|\widehat{\mu}(z-y)\big|\big|\widehat{\varphi}(y)\big|^{\frac{\theta}{2}} \big|\widehat{\varphi}(y)\big|^{\frac{2-\theta}{2}}\,dy\bigg]^{\frac{2}{\theta}}\\
		&\leq \Bigg[ \bigg( \int_{\rd}\big| \widehat{\mu}(z-y) \big|^{\frac{2}{\theta}} \big| \widehat{\varphi}(y) \big|\,dy\bigg)^{\frac{\theta}{2}}\bigg( \int_{\rd} \big| \widehat{\varphi}(y) \big|\,dy \bigg)^{\frac{2-\theta}{2}} \Bigg]^{\frac{2}{\theta}}\\
		&= \int_{\rd} \big| \widehat{\mu}(z-y) \big|^{\frac{2}{\theta}}\big| \widehat{\varphi}(y) \big|\,dy \,\bigg(\int_{\rd} \big| \widehat{\varphi}(y) \big|\,dy \bigg)^{\frac{2-\theta}{\theta}}\\
		&\lesssim \int_{\rd}\big| \widehat{\mu}(z-y) \big|^{\frac{2}{\theta}}\big| \widehat{\varphi}(y) \big|\,dy = \big( \big|\widehat{\varphi}\big| * \big|\widehat{\mu}\big|^{\frac{2}{\theta}} \big)(z).
	\end{align*}	
	 By \eqref{eq:projmeasure}, $\widehat{\mu_{V}}(y) = \widehat{\mu}(y_{V})$, and so  we have the following estimate:
	\begin{align*}
		\int_{\gdk}\J_{u,\theta}(\mu_{V})^{1/\theta}\,d\nu(V) &= \int_{\gdk}\int_{\rk} \big| \widehat{\mu_{V}}(y) \big|^{\frac{2}{\theta}}|y|^{\frac{u}{\theta}-k} \, dy \,d\nu(V)\\
		&= \int_{\gdk}\int_{\rk}\big| \widehat{\mu}(y_{V}) \big|^{\frac{2}{\theta}}|y|^{\frac{u}{\theta}-k}\, dy \,d \nu(V)\\
		&\lesssim \int_{\gdk}\int_{\rk}\big( \big|\widehat{\varphi}\big|*\big|\widehat{\mu}\big|^{\frac{2}{\theta}} \big)(y_{V})|y|^{\frac{u}{\theta}-k}\, dy \,d\nu(V)\\
		&= \int_{\gdk}\int_{\rk}\bigg( \int_{\rd}\big| \widehat{\varphi}(y_{V}-z) \big| \big| \widehat{\mu}(z) \big|^{\frac{2}{\theta}}\,dz\bigg)|y|^{\frac{u}{\theta}-k}\, dy \,d\nu(V)\\
		&= \int_{\rd}\big| \widehat{\mu}(z) \big|^{\frac{2}{\theta}}\Bigg[\int_{\gdk}\int_{\rk} \big| \widehat{\varphi}(y_{V}-z) \big||y|^{\frac{u}{\theta}-k} \, dy \,d\nu(V)\Bigg]\,dz\\
		&\lesssim \int_{\rd}\big| \widehat{\mu}(z) \big|^{\frac{2}{\theta}}\Bigg[ \int_{\gdk}\int_{\rk}\big(1+|y_{V}-z|\big)^{-N}|y|^{\frac{u}{\theta}-k}\, dy \,d\nu(V) \Bigg]\,dz.
	\end{align*}
	
	To finish the proof of the theorem we need to show that the last integral is finite. It is enough to see that for $N$ sufficiently large,
	\begin{equation}\label{eq:thm_bound}
		\int_{\gdk}\int_{\rk}\big(1+|y_{V}-z|\big)^{-N}|y|^{\frac{u}{\theta}-k}\, dy \,d\nu(V)\lesssim |z|^{\frac{u}{\theta}+k(d-k)-d-\tau},
	\end{equation}
for all $z \in \rd$ with $|z| \geq 2$ because then, since $k(d-k)+\frac{u-s}{\theta}<\tau$,
	\begin{align*}
		\int_{\gdk}\J_{u,\theta}(\mu_{V})^{1/\theta}\,d\nu(V) 
		&\lesssim \int_{\rd} \big| \widehat{\mu}(z) \big|^{\frac{2}{\theta}}|z|^{\frac{u}{\theta}+k(d-k)-d-\tau}\,dz \\
		&\lesssim \int_{\rd}\big| \widehat{\mu}(z) \big|^{\frac{2}{\theta}}|z|^{\frac{s}{\theta}-d}\,dz\\
		&= \J_{s,\theta}(\mu)^{1/\theta}<\infty,
	\end{align*}
	since $s<\fs \mu$.  This establishes \eqref{eq:thm1_energy} and completes the proof.
	
	To prove \eqref{eq:thm_bound} note that from the definition of $\nu$ we have for all $r>0$ and $z\in \rd$,
	\begin{equation} \label{frostmancondition}
		\nu\big( \{ V\in\gdk : d(z,V)\leq r \} \big)\lesssim \bigg( \frac{r}{|z|}\bigg)^{\tau-(k-1)(d-k)}
	\end{equation}
where here and in what follows,  $d(z,Y) = \inf\{ |z-y| :  y \in Y\}$. Fix $z \in \rd$ with $|z| \geq 2$ and  split the integral into the dyadic annuli centred at $z$ as follows:
	\begin{align*}
		\int_{\gdk}\int_{\rk}\big(1+|y_{V} - z|\big)^{-N}&|y|^{\frac{u}{\theta}-k}\, dy \,d\nu(V) \\
    &\quad =\iint_{\{(V,y) : |y_{V}-z|\leq 1/2 \}} + \sum_{\{ j \geq 0 : |z|> 2^{j+1} \}}\iint_{\{ (V,y) : 2^{j-1}<|y_{V}-z|\leq 2^{j} \}} \\
    &\quad \quad + \sum_{\{ j \geq 0 : |z|\leq 2^{j+1} \}}\iint_{\{ (V,y) : 2^{j-1}<|y_{V}-z|\leq 2^j \}}
	\end{align*}
where the sums are over integer $j$.  For clarity, let us analyse each of the terms separately.\\
	
	\textbf{First term:} In this case we have $|y|\approx |z|$, the inclusion
	\begin{equation*}
		\{ (V,y) : |y_{V}-z|\leq 1/2 \}\subseteq\{ (V,y) : d(z,V) \leq 1/2 \},
	\end{equation*}
	and the trivial estimate 
	\begin{equation*}
		\int_{\rk}\big(1+|y_{V}-z|\big)^{-N}\, dy   \lesssim 1
	\end{equation*}
which holds for $N$ sufficiently large. These three things yield
	\begin{align*}
		\iint_{\{ (V,y) : |y_{V}-z|\leq 1/2 \}}&\big(1+|y_{V}-z|\big)^{-N}|y|^{\frac{u}{\theta}-k}\, dy \,d\nu(V)\\
		&\lesssim |z|^{\frac{u}{\theta}-k}\int_{\{ V : d(z,V)\leq 1/2 \}}\int_{\rk} \big(1+|y_{V}-z|\big)^{-N}\, dy \,d\nu(V)\\
		&\lesssim |z|^{\frac{u}{\theta}-k}\nu\big( \{ V : d(z,V)\leq 1/2 \} \big)\\
		&\lesssim |z|^{\frac{u}{\theta}-k-\tau+(k-1)(d-k)}\\
		&= |z|^{\frac{u}{\theta}+k(d-k)-d-\tau} \numberthis\label{eq:firstcase}
	\end{align*}
 by \eqref{frostmancondition}.
	
	\textbf{Second term:} Since $|y_{V}-z|\leq|z|/2$, we again have  $|y|\approx |z|$. 	Then
	\begin{align*}
		\sum_{\{ j \geq 0 : |z|>2^{j+1} \}}&\iint_{\{ (V,y): 2^{j-1}<|y_{V}-z|\leq 2^j \}} \big(1+|y_{V} - z|\big)^{-N}|y|^{\frac{u}{\theta}-k}\, dy \,d\nu(V) \\
		&\quad \lesssim |z|^{\frac{u}{\theta}-k}\sum_{\{ j \geq 0 : |z|>2^{j+1} \}} 2^{-jN}\iint_{\{(V,y) : 2^{j-1}<|y_{V}-z|\leq 2^j \}}\, dy \,d\nu(V)\\
		&\quad\lesssim |z|^{\frac{u}{\theta}-k}\sum_{\{ j\geq 0 : |z|>2^{j+1} \}}2^{-jN}2^{kj}\nu\big( \{ V : d(z,V)\leq 2^j \} \big)\\
		&\quad\lesssim |z|^{\frac{u}{\theta}-k-\tau+(k-1)(d-k)}\sum_{j=0}^\infty2^{j(k+\tau-(k-1)(d-k)-N)} \qquad \text{(by \eqref{frostmancondition})}\\
		&\quad\lesssim |z|^{\frac{u}{\theta}+k(d-k)-d-\tau} \numberthis\label{eq:secondcase}
	\end{align*}
provided  $N>k+\tau-(k-1)(d-k) = d+\tau-k(d-k)$.
	
	\textbf{Third term:}   Using that, for each relevant $j$,  $|y| \leq 2^j+|z| \leq 3 \cdot 2^{j}$ and $\nu(\gdk) \lesssim 1$,
	\begin{align*}
		\sum_{\{ j\geq 0 : |z|\leq 2^{j+1} \}}&\iint_{\{ (V,y): 2^{j-1}<|y_{V}-z|\leq 2^j \}}\big( 1+|y_{V}-z| \big)^{-N}|y|^{\frac{u}{\theta}-k}\, dy \,d\nu(V)\\
		& \lesssim\sum_{\{ j\geq 0: |z|\leq2^{j+1} \}}2^{-jN}\iint_{\{ (V,y) : 2^{j-1}<|y_{V}-z|\leq 2^j \}}|y|^{\frac{u}{\theta}-k}\, dy \,d\nu(V)\\
		& \lesssim \sum_{\{ j\geq 0 : |z|\leq 2^{j+1} \}}2^{-jN}\int_{\gdk}\int_{\{y : |y|\leq 3 \cdot 2^j \}}|y|^{\frac{u}{\theta}-k}\, dy \,d\nu(V)\\
		& \lesssim \sum_{\{ j\geq 0 : |z|\leq 2^{j+1} \}} 2^{-jN}\int_{\{y : |y|\leq 3 \cdot 2^j\}}|y|^{\frac{u}{\theta}-k}\,  dy \\
		& \lesssim\sum_{\{ j\geq 0 : |z|\leq 2^{j+1} \}}2^{j\big( \frac{u}{\theta}-N \big)}\\
	& \lesssim|z|^{ \frac{u}{\theta}-N }\\
& \leq |z|^{\frac{u}{\theta}+k(d-k)-d-\tau} \numberthis\label{eq:thirdcase}
	\end{align*}
	provided $N > \max\{ d+\tau-k(d-k), u/\theta \}$.  Combining  \eqref{eq:firstcase}, \eqref{eq:secondcase} and \eqref{eq:thirdcase} we have the theorem.
\end{proof}

\section{Applications}\label{section:applications}

\subsection{Bounding the exceptional set for Hausdorff dimension}\label{sec:improv}

Let $X\subseteq\rd$ be a Borel set and $\mu\in\M(\rd)$. Since for all $\theta\in[0,1]$, $\min\{ d, \fs \mu \}\leq \hd \mu$ and $\fs X\leq \hd X$, we can use Theorem~\ref{thm:exceptionalFourier} to bound the exceptional set for the \emph{Hausdorff} dimension.  This corollary can be viewed as our main result, even though Theorem~\ref{thm:exceptionalFourier} is stronger.

\begin{cor}\label{thm:exceptionalFouriercoro}
	Let $\mu\in\M(\rd)$, $\theta\in(0,1]$ and $1\leq k<d$ be an integer. Then for all $u\in[0,k]$,
	\begin{equation*}
		\hd \{ V\in\gdk : \hd \mu_{V}<u \}\leq\max\bigg\{0, k(d-k)+\inf_{\theta\in(0,1]}\frac{u-\fs  \mu}{\theta} \bigg\}.
	\end{equation*}
	Furthermore, if $X$ is a Borel set in $\rd$ and $\theta\in(0, 1]$, then for all $u\in[0,k]$,
	\begin{equation*}
		\hd \{ V\in \gdk : \hd  P_{V}(X)<u \}\leq \max\bigg\{0, k(d-k)+\inf_{\theta\in(0,1]}\frac{u-\fs  X}{\theta} \bigg\}.
	\end{equation*}
\end{cor}

Writing the $n$-fold convolution as $\mu^{*n}$  for  $n\in\N$, we obtain  the following corollary as a direct application of \cite[Lemma~6.2]{Fra22}.  This formulates our main result above for measures in terms of the Sobolev dimension of  convolutions.

\begin{cor}\label{thm:exceptionalFourierConvolutions}
	Let $\mu\in\M(\rd)$, $n\in\N$ and $1\leq k<d$ be an integer. Then for all $u\in[0,k]$,
	\begin{equation*}
		\hd \{ V\in\gdk : \hd \mu_{V}<u \}\leq\max\big\{0, k(d-k)+\inf_{n\in\N}(nu-\sd(\mu^{*n})) \big\}.
	\end{equation*}
\end{cor}

This corollary will be useful in Section~\ref{sec:example}, where we calculate the Fourier spectrum of an explicit self-affine measure at $\theta = 1/2$ and use it to bound the exceptional set of projections for the Hausdorff dimension of a self-affine set. That example shows that the Fourier spectrum can be used to improve on the state of the art bounds in higher dimensions in a representative situation, which we discuss in more detail in the next subsections.

\subsubsection{Improving the sharp bounds in dimension 2}

With  Corollary \ref{thm:exceptionalFouriercoro} in mind, and specialising only to the case of Borel sets $X \subseteq \rr$, we ask for conditions under which we are able to improve Ren--Wang's sharp inequality \cite[Theorem~1.2]{RW23}; recall \eqref{eq:oberlin}.  Indeed, this will happen  for $X\subseteq \R^2$ provided  
\begin{equation}\label{eq:improvingRW}
  1 + \frac{u - \fs X}{\theta}< 2u-\hd X
\end{equation}
for some $\theta\in(0,1]$ and some    $u$ in the range $ \frac{\hd X}{2} < u \leq \min\{  \hd X, 1 \}$. 

First consider the case when $\hd X<1$, see Figure~\ref{fig:RWImprovement}: Left. Since $\fs X \leq \hd X$, \eqref{eq:improvingRW} is only possible for $\theta<\frac{1}{2}$ in which case we require
\begin{equation*}
  \fs X>\frac{\hd X}{2} +\theta.
\end{equation*}
This would then give improvements on \eqref{eq:oberlin} for $u$ satisfying
\[
\frac{\hd X}{2} < u< \min\left\{\frac{\fs X - \theta(1+\hd X)}{1-2\theta}, \hd X\right\}.
\]

Next consider the case when $\hd X\geq1$, see Figure~\ref{fig:RWImprovement}: Right.  For $\theta<\frac{1}{2}$, we (again) require 
\begin{equation}\label{eq:slope1second}
  \fs X>\frac{\hd X}{2} +\theta
\end{equation}
and this  gives  improvements on \eqref{eq:oberlin} for $u$ satisfying
\[
\frac{\hd X}{2} < u< \min\left\{\frac{\fs X - \theta(1+\hd X)}{1-2\theta}, 1\right\}.
\]
On the other hand, for  $\theta>\frac{1}{2}$ we require
\[
 \fs X > 1+ \theta(\hd X-1)
\]
and this gives  improvements on \eqref{eq:oberlin} for $u$ satisfying
\[
\max\left\{\frac{  \theta(1+\hd X)-\fs X}{2\theta-1}, \frac{\hd X}{2}\right\}  < u< 1.
\]
Note   that this interval will be empty if  
\begin{equation}\label{eq:problem}
\fs X \geq \frac{\hd X}{2}+\theta.
\end{equation}
However,  since $\fs X$ is continuous,  non-decreasing and bounded above by $\hd X$, if \eqref{eq:problem} holds for some $\theta >\frac{1}{2}$ then it fails for $\theta >\tfrac{\hd X}{2}$, and we still get improvement on \eqref{eq:oberlin} by using a different $\theta$.

Curiously, for $\theta = \frac{1}{2}$ we get improvement for all $u$ as long as
\begin{equation}\label{eq:improv1half}
    \fd^{1/2}X >\frac{1+\hd X}{2}.
\end{equation}

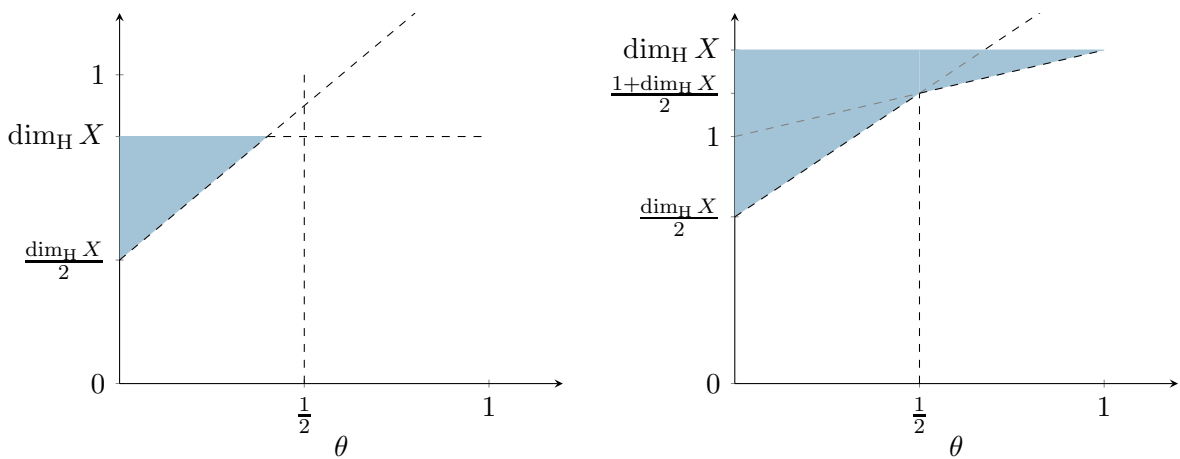
\begin{figure}[H]
  \centering
  \begin{subfigure}{.5\textwidth}
    \centering
    \begin{tikzpicture}[scale = 0.7]
      \def\dimX{0.8}
      \begin{axis}[
          axis lines = left,
          ymin = 0,
          ymax = 1.2,
          xmin = 0,
          xmax = 1.2,
          xlabel=$\theta$,
          xtick = {0.5,1},
          ytick = {0,\dimX/2,\dimX,1},
          xticklabels = {$\frac{1}{2}$, $1$},
          yticklabels = {$0$,$\frac{\hd X}{2}$, $\hd X$, $1$}
      ]
      \addplot [
          domain=-0:\dimX/2,
          samples=100, 
          color=plotblue,
          name path=dimX
      ]
      {\dimX};
      \addplot [
          domain=\dimX/2:1,
          style=dashed,
          samples=100, 
          color=black
      ]
      {\dimX};
      \addplot [
          domain=0:0.8, 
          samples=100, 
          style=dashed,
          color=black,
          name path=lower
      ]
      {\dimX/2 +x};
      \addplot[plotblue] fill between[of=lower and dimX, soft clip={domain=0:\dimX/2}];
      \addplot[
        domain=0:6,
        style=dashed,
      ] coordinates {(0.5,0)(0.5,1)};
      \end{axis}
    \end{tikzpicture}
  \end{subfigure}%
  \begin{subfigure}{.5\textwidth}
    \centering
    \begin{tikzpicture}[scale = 0.7]
      \def\dimX{1.35}
      \begin{axis}[
          axis lines = left,
          ymin = 0,
          ymax = 1.5,
          xmin = 0,
          xmax = 1.2,
          xlabel=$\theta$,
          xtick = {0.5,1},
          ytick = {0,\dimX/2, 1, (1+\dimX)/2,\dimX},
          xticklabels = {$\frac{1}{2}$, $1$},
          yticklabels = {$0$,$\frac{\hd X}{2}$, $1$, $\frac{1+\hd X}{2}$,$\hd X$}
      ]
      \addplot [
          domain=0:0.5,
          samples=100, 
          color=plotblue,
          name path=dimX1
      ]
      {\dimX};
      \addplot [
          domain=0.5:0.68,
          samples=100, 
          color=plotblue,
          name path=dimX2
      ]
      {\dimX};
      \addplot [
          domain=0.68:1,
          samples=100, 
          color=plotblue,
          name path=dimX
      ]
      {\dimX};
      \addplot[
        domain=0:1,
        style=dashed,
      ] coordinates {(0.5,0)(0.5,0.5 +\dimX/2)};
      \addplot [
          domain=0:0.5, 
          samples=100, 
          style=dashed,
          name path=lower1
      ]
      {\dimX/2 +x};
      \addplot [
          domain=0.5:0.68, 
          samples=100, 
          color=gray,
          style=dashed,
          name path=lower2
      ]
      {\dimX/2 +x};
      \addplot [
          domain=0.68:1, 
          samples=100, 
          color=black,
          style=dashed,
      ]
      {\dimX/2 +x};
      \addplot [
          domain=0:0.5, 
          samples=100, 
          color=gray,
          style=dashed,
          name path=upper1
      ]
      {1+x*(\dimX-1)};
      \addplot [
          domain=0.5:1, 
          samples=100, 
          style=dashed,
          color=black,
          name path=upper2
      ]
      {1+x*(\dimX-1)};
      
      \addplot[plotblue] fill between[of=lower1 and dimX1, soft clip={domain=0:1}];
      \addplot[plotblue] fill between[of=upper2 and dimX2, soft clip={domain=0.5:1}];
      \end{axis}
    \end{tikzpicture}
  \end{subfigure}
  \caption{In order to improve Ren--Wang's exceptional set estimate for a Borel set $X \subseteq \rr$, we need the Fourier spectrum of $X$ to intersect the  shaded region. Left: when $\hd X<1$. Right: when $\hd X\geq1$.}
  \label{fig:RWImprovement}
  \end{figure}

It is easily seen in Figure~\ref{fig:RWImprovement} that if   $\fd X>\hd X/2$, then we get improvement on \eqref{eq:oberlin}, but we already knew this from the discussion in Section \ref{section:FDproj}.  However, what Figure~\ref{fig:RWImprovement} now reveals is that this is not needed and, in fact, improvement can be gained for sets with small Fourier dimension but large Fourier spectrum.  Important to observe here is that the slope of the lines bounding the shaded region is never more than 1, whereas the slope of the Fourier spectrum of a subset of the plane can be as large as 2, see \cite{CFdO24} (that is, the Fourier spectrum can `catch up' even if the Fourier dimension is too small to give improvement by itself).  It is more or less clear that  examples admitting improvement of \eqref{eq:oberlin} via the Fourier spectrum abound.   That said, building explicit examples is not completely straightforward, partly  because it is difficult to compute the Fourier spectrum in specific settings.

Another interesting feature of Figure~\ref{fig:RWImprovement} is that if $\hd X \leq  1$, then the Fourier spectrum can only be used to improve \eqref{eq:oberlin} if $\theta < \hd X/2 \leq 1/2$ but if  $\hd X \geq  1$, then the whole range of $\theta$ is potentially relevant. One reason for this could be that in the case $\hd X<1$, the bound given in Theorem~\ref{thm:exceptionalFourier} for $\theta=1$ is not sharp, as was mentioned in Section~\ref{section:preliminariesProj}. This is one of the reasons why it would be interesting to obtain a Fourier spectrum analogue of the sharp bounds. 

By scrutinising the limit as $\theta \to 1$ we get a useful corollary. 

\begin{cor}
 Let $X \subseteq \rr$ be a Borel set with $\hd X >1$ and 
\[
D = \overline{\partial_{-}}\fs X\big|_{\theta =1}  = \limsup_{\theta\to1} \frac{\hd X - \fs X}{1-\theta} 
\]
be the upper left semi-derivative of $\fs X$ at $\theta = 1$.  If $D <\hd X -1$, then there is necessarily improvement on Ren--Wang's inequality \eqref{eq:oberlin}.
\end{cor}

\subsubsection{Improving the known  bounds in higher dimensions}

We now consider the higher dimensional case, where the sharp bounds are not known.  Let $X\subseteq\rd$, for $d\geq3$, be a Borel set and $1\leq k<d$ be an integer. We can use Corollary~\ref{thm:exceptionalFouriercoro} to improve Mattila's \eqref{eq:mattilaboundgeneral} and Peres--Schlag's \eqref{eq:peresschlagbound} when $\hd X\leq k$ and $\hd X>k$, respectively.

Let us first consider the case $\hd X\leq k$; see Figure \ref{fig:PSImprovement}: Left. Then we want to have for some $\theta\in(0,1]$ and some $0<u\leq\hd X$,
\begin{equation*}
    k(d-k) + \frac{u-\fs X}{\theta} < k(d-k-1) + u,
\end{equation*}
which yields
\begin{equation*}
    u(1-\theta) + k\theta< \fs X.
\end{equation*}
Note that no improvement will be possible when $u = \hd X$, that is, when \eqref{eq:mattilaboundgeneral} is sharp.  However, for smaller values of $u$ (especially for $u$ close to zero) improvement is possible, and even quite common; see Corollaries \ref{genericimprove} and \ref{genericimprove2}.


Similarly, in the case $\hd X \geq k$ we can improve Peres--Schlag's bound \eqref{eq:peresschlagbound}; see Figure \ref{fig:PSImprovement}: Right. This would happen if
\begin{equation}
  k(d-k) + \frac{u-\fs X}{\theta} < k(d-k) + u - \hd X,
\end{equation}
for some $\theta\in(0,1]$ and some $u$ in the range $0\leq u\leq k$. This gives
\begin{equation}\label{eq:improvePS}
    u(1-\theta)+ \theta\hd X < \fs X.
\end{equation}
This time improvement is possible for any relevant value of $u$. In fact, it turns out that  the Fourier spectrum will lead to an improvement over the known bounds for  `most' sets in some sense.  This is captured by the following corollary.

  \begin{cor} \label{genericimprove}
    Let $X \subseteq \rd$ be a Borel set.     If $\fs X  >\theta \max\{k,\hd X\} $ for some $\theta \in [0,1)$ then there is necessarily improvement on Peres--Schlag's inequality \eqref{eq:peresschlagbound} and Mattila's inequality \eqref{eq:mattilaboundgeneral} for all sufficiently small $u>0$.
   \end{cor}

 Important to note in the above is  that it \emph{always} holds that $\fs X  \geq \theta \hd X$; see \cite[Theorem 1.1]{Fra22}.  We now show that one only has to impose very mild conditions on $X$ in order to get $\fs X  >\theta \max\{k,\hd X\} $ for all sufficiently small $\theta$.  First of all, note that we get strict inequality close to $\theta=0$  immediately if the Fourier dimension is strictly positive and so we focus on the case where the Fourier dimension is zero.   We use recent work of Khalil \cite{Kha23} to get good information about the Fourier spectrum.

Following \cite{Kha23}, a Borel measure $\mu$ on $\rd$ is \textit{$(C,\alpha)$-uniformly affinely non-concentrated} if it satisfies that there exists $C\geq 1$ and $\alpha>0$ such that for every $\varepsilon>0$, $x\in\rd$, $0<r\leq 1$ and  hyperplane $W\subseteq\rd$,
\begin{equation}\label{eq:Khalil}
  \mu\big( W^{(\varepsilon r)}\cap B(x,r) \big)\leq \varepsilon^{\alpha}\mu\big( B(x,r) \big),
\end{equation}
where $W^{(r)}$ is the $r$-neighbourhood of $W$.  This condition will be satisfied very generally.  For example, when $d=1$ asking that $X \subseteq \mathbb{R}$ supports such a measure   is simply asking for $X$ to be uniformly perfect.  In higher dimensions the condition is satisfied widely for IFS attractors not lying in a hyperplane (including self-similar sets, self-affine sets, self-conformal sets) and many other sets including Mandelbrot percolation, Julia sets etc.  Many of these sets may have Fourier dimension zero (arithmetic self-similar sets,  self-affine sponges etc) but it turns out that even when the Fourier dimension is zero, the Fourier spectrum will be `as large as possible' near $\theta = 0$.

\begin{lma}\label{lma:khalil}
  If $X\subset\rd$ is a compact set with $\fd X = 0$ which  supports a $(C, \alpha)$-uniformly affinely non-concentrated Borel measure, then
\[
\limsup_{\theta \to 0} \frac{ \fs X} {\theta} = d.
\]
\end{lma}
 \begin{proof}
Let $\mu$ be a $(C, \alpha)$-uniformly affinely non-concentrated Borel measure on $X$.  Then  \cite[Theorem~1.6]{Kha23} gives that the Frostman dimension of the $n$-fold convolution $\mu^{*n}$ converges to $d$ as $n \to \infty$.  However, since $\fd \mu \leq \fd X = 0$, it follows from   \cite[Corollary~1.8]{Fra22} and the fact that the Frostman dimension is a lower bound for the Sobolev dimension, that  $\sd(\mu^{*n}) \to d$ as $n \to \infty$ also.  Moreover, by  \cite[Lemma~6.2]{Fra22} $\sd(\mu^{*n}) = n\fd^{1/n}\mu$, so we get
\begin{equation*}
\limsup_{\theta \to 0} \frac{ \fs X} {\theta} \geq  \limsup_{n\to\infty} \frac{\fd^{1/n}\mu}{1/n} = \limsup_{n\to\infty}\sd(\mu^{*n}) = d.
\end{equation*}
Recall that by \cite[Proposition~4.2]{CFdO24}, the reverse inequality also holds.
 \end{proof}

Combining Corollary \ref{genericimprove} and Lemma \ref{lma:khalil} we get the following, which provides improvement on the best known exceptional set estimates very generally.

  \begin{cor} \label{genericimprove2}
    Let $X \subseteq \rd$ be a compact set with $0<\hd X<d$ and which     supports a $(C, \alpha)$-uniformly affinely non-concentrated Borel measure.  Then the Fourier spectrum provides an  improvement on Peres--Schlag's inequality \eqref{eq:peresschlagbound} and Mattila's inequality \eqref{eq:mattilaboundgeneral} for all sufficiently small $u>0$.
   \end{cor}

\begin{figure}[H]
  \begin{subfigure}{.5\textwidth}
    \centering
    \begin{tikzpicture}[scale = 0.7]
      \def\dimX{1}
      \def\k{1.35}
      \def\u{0.4}
      \begin{axis}[
          axis lines = left,
          ymin = 0,
          ymax = 1.5,
          xmin = 0,
          xmax = 1.2,
          xlabel=$\theta$,
          xtick = {1},
          ytick = {0,\u, \k,\dimX},
          xticklabels = {$1$},
          yticklabels = {$0$,$u$,$k$,$\hd X$}
      ]
      \addplot [
          domain=0:0.63,
          samples=100, 
          color=gray,
          style=dashed,
          name path=dimX
      ]
      {\dimX};
      \addplot [
          domain=0.63:1,
          samples=100, 
          color=black,
          style=dashed,
          name path=dimX2
      ]
      {\dimX};
      \addplot [
          domain=0:1,
          samples=100, 
          color=black,
          style=dashed,
          name path=ux
      ]
      {\u*(1-x) + \k*x};
      \addplot [
          domain=0:1,
          samples=100, 
          color=black,
          style=dashed,
          name path=kx
      ]
      {\k};
      \addplot[plotblue] fill between[of=ux and dimX, soft clip={domain=0:0.63}];
      
      \end{axis}
    \end{tikzpicture}
  \end{subfigure}%
  \begin{subfigure}{.5\textwidth}
    \centering
    \begin{tikzpicture}[scale = 0.7]
      \def\dimX{1.35}
      \def\k{0.8}
      \def\u{0.4}
      \begin{axis}[
          axis lines = left,
          ymin = 0,
          ymax = 1.5,
          xmin = 0,
          xmax = 1.2,
          xlabel=$\theta$,
          xtick = {1},
          ytick = {0,\u, \k,\dimX},
          xticklabels = {$1$},
          yticklabels = {$0$,$u$,$k$,$\hd X$}
      ]
      \addplot [
          domain=0:1,
          samples=100, 
          color=plotblue,
          name path=dimX
      ]
      {\dimX};
      \addplot [
          domain=0:1,
          samples=100, 
          color=black,
          style=dashed,
          name path=ux
      ]
      {\u*(1-x) + \dimX*x};
      \addplot [
          domain=0:1,
          samples=100, 
          color=gray,
          style=dashed,
          name path=kx
      ]
      {\k*(1-x) + \dimX*x};
      \addplot[plotblue] fill between[of=ux and dimX, soft clip={domain=0:1}];
      
      \end{axis}
    \end{tikzpicture}
  \end{subfigure}
  \caption{In order to improve the higher-dimensional exceptional set estimates for a Borel set $X \subseteq \rd$ at $u$, we need the Fourier spectrum of $X$ to intersect the  shaded region. Left: when $\hd X\leq k$. Right: when $\hd X> k$.}
  \label{fig:PSImprovement}
  \end{figure}
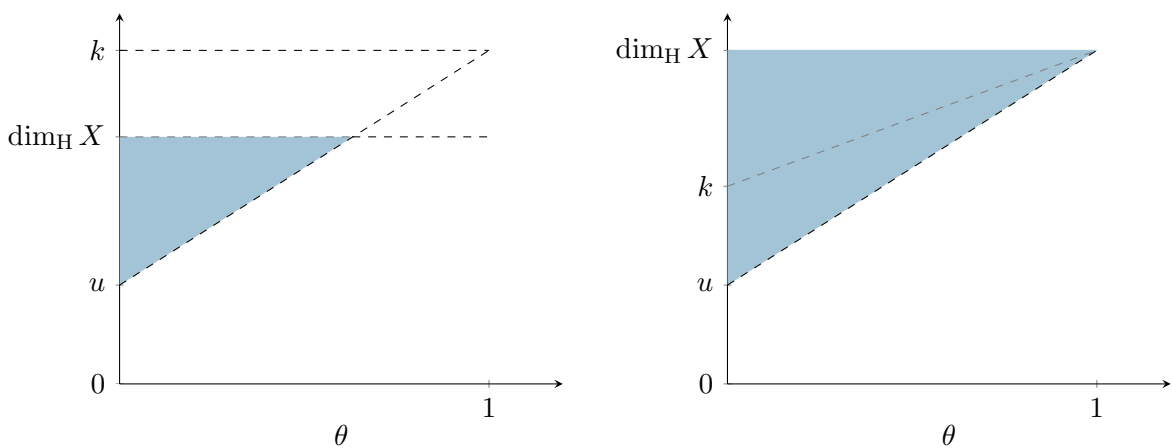

 Finally, considering large $\theta$, we obtain the following corollary taking the limit as $\theta\to1$

  \begin{cor}
    Let $X \subseteq \rd$ be a Borel set with $\hd X >k$ and 
   \[
   D = \overline{\partial_{-}}\fs X\big|_{\theta =1}  = \limsup_{\theta\to1} \frac{\hd X - \fs X}{1-\theta},
   \]
   be the upper left semi-derivative of $\fs X$ at $\theta = 1$.  If $D <\hd X -u$ for some $u\in[0,k]$ then there is necessarily improvement on Peres--Schlag's inequality \eqref{eq:peresschlagbound}.
   \end{cor}

\subsection{Dimension bounds for the set of projections with empty interior}

Following \cite[Corollary 6.2]{PS00}, we use the previous exceptional set estimate to obtain results regarding the interior of projections. Note first that if $\fs X >2k$ for some $\theta\in[0,1]$, then $\hd X >2k$ and by \cite[Corollary~5.12~(d)]{Mat15}, $P_{V}(X)$ has non-empty interior for $\gamdk$ almost all $V\in\gdk$.

\begin{cor}
  Let $X\subseteq\rd$ be a Borel set and $1\leq k<d$ be an integer. If $\fs X >2k$ for some $\theta\in(0,1]$, then
  \begin{equation*}
    \hd\{ V\in\gdk : P_{V}(X)~\text{has empty interior} \}\leq k(d-k) + \inf_{\theta\in(0,1]}\frac{2k - \fs X}{\theta} < k(d-k).
\end{equation*}
\end{cor}
\begin{proof}
  Fix $\theta \in (0,1]$,  let  $\varepsilon>0$ and choose $\mu \in \mathcal{M}(X)$ such that $\fs \mu \geq \fs X - \varepsilon$. We will use the following inclusion of sets, which follows from the fact that if $\fs\mu_{V}>2k$, then $\widehat{\mu_{V}}\in L^2(\rk)$ and so $\mu_{V}$ is continuous and $P_V(X)$ has non-empty interior. Therefore,
  \begin{align*}
    \{ V\in \gdk : P_{V}(X)~\text{has empty interior} \}&\subseteq \{ V\in \gdk : \mu_{V}~\text{not continuous} \}\\
    &\subseteq \{ V\in \gdk : \fs\mu_{V}\leq 2k \}\\
    &\subseteq \{ V\in \gdk : \fs\mu_{V}<2k+\varepsilon \}.
  \end{align*}
  Theorem~\ref{thm:exceptionalFourier} gives an upper bound for the dimension of the final set in this chain of inclusions and therefore, since $\theta\in(0,1]$ was arbitrary, 
  \begin{align*}
      \hd \{ V\in \gdk : P_{V}(X)~\text{has empty interior} \} &\leq k(d-k) + \inf_{\theta\in(0,1]}\frac{2k+\eps-\fs\mu}{\theta}\\
&\leq k(d-k) + \inf_{\theta\in(0,1]}\frac{2k+2\eps-\fs X}{\theta}
  \end{align*}
and letting $\varepsilon \to 0$  proves the result.
\end{proof}

\subsection{Continuity of the dimension of the exceptional set}

Proposition~\ref{propo:counter_example} showed us that for sets $X$, the dimension of the set of  exceptional directions can be discontinuous at $u = \fd X$.  One of  the questions which motivated our investigation in the first place was to determine conditions under which continuity of the dimension of the set of  exceptional directions could be recovered at  $u = \fd X$.   We show in the following proposition that such a condition can be given in terms of the Fourier spectrum.  There is an analogous result for measures, which we leave to the reader to formulate.  

\begin{prop}\label{prop:continuity}
Let $X$ be a Borel set in $\rd$ and let
\[
D = \underline{\partial_{+}}\fs X\big|_{\theta =0} = \liminf_{\theta \to 0} \frac{\fs X - \fd X}{\theta}
\]
be the lower right semi-derivative of $\fs X$ at $\theta = 0$.  If $D \geq k(d-k)$, then the function  $u\mapsto\hd \{ V\in\gdk : \hd P_{V}(X)<u \}$ is continuous at $u = \fd X$.
\end{prop}
\begin{proof}
Let $\varepsilon\in(0,1)$  and consider $u = \fd X + \varepsilon^2$.  Corollary \ref{thm:exceptionalFouriercoro} gives that  
	\begin{align*}
		\hd \big\{ V\in \gdk : \hd  P_{V}(X)<& \fd X + \varepsilon^2 \big\}\\
    &\leq \max\bigg\{0, k(d-k)+\inf_{\theta\in(0,1]}\frac{ \fd X + \varepsilon^2-\fs  X}{\theta} \bigg\}\\
&\leq \max\bigg\{0, k(d-k)+  \varepsilon-\frac{ \fd^\varepsilon X  -\fd  X}{\varepsilon} \bigg\}\\
& \to 0
	\end{align*}
as $\varepsilon \to 0$ provided $D \geq k(d-k)$, which proves the desired continuity result. 
\end{proof}

We know by \cite[Proposition~4.2]{CFdO24} that for any Borel set $X\subseteq\rd$, $\underline{\partial_{+}}\fs X\big|_{\theta =0}\leq d$. Therefore,  in order to  establish continuity of the dimension of the exceptional set from  Proposition~\ref{prop:continuity}, it is necessary to have $k(d-k)\leq d$.  This is only true for the families $G(d,1)$, $G(d,d-1)$, and the special case $G(4,2)$.

In Proposition~\ref{propo:counter_example} we built a non-Salem set $X\subseteq\rr$ for which the dimension of the exceptional set was discontinuous at $\fd X$. However, like Salem sets, the set $X$ satisfies $\partial_{+}\fs X\big|_{\theta=0}=0$. To see this, recall that $X$  was the union of a set $A$ coming from Lemma~\ref{lemma:sharp1} and  a Salem set $B$.  With a little more work, one can show that $\fd A < \fd B$. Thus, there must exist  $\lambda\in(0,1)$ such that $\fd^{\lambda}A = \fd B$ and then for all $\theta\leq \lambda$, since $X = A\cup B$, $\fs X = \fs B$, which gives $\partial_{+}\fs X\big|_{\theta=0}=0$. This raises the following question.

\begin{ques}
  Is $\underline{\partial_{+}}\fs X\big|_{\theta=0}>0$ sufficient to guarantee continuity of the dimension of the exceptional set at $u = \fd X$? Or perhaps  $\underline{\partial_{+}}\fs X\big|_{\theta=0}\geq \rho$ for some $\rho>0$? (We know from the above that $\rho= k(d-k)$ suffices.)
\end{ques}

\section{An explicit example}\label{sec:example}

Here we provide an explicit example where we exploit  dynamical structure to get good information about  the Fourier spectrum  at $\theta=1/2$.  We feed this information into Corollary \ref{thm:exceptionalFouriercoro} and use this to showcase how our results compare with the best known general bounds in a representative situation.    In general, computing the Fourier spectrum explicitly is difficult and so our example is by necessity quite simple.  In particular, one may attempt to bound the dimension of the exceptional set directly using the specific structure at hand or by directly quoting the extensive literature on self-similar and self-affine sets---and probably obtain much stronger results---but we do not pursue this here. The  example  we consider is a self-affine set $X$  in $\mathbb{R}^3$ built as the product of three self-similar sets in the line with different contraction ratios. We design the example such that the sumset  $X+X$ is itself a  self-affine  set  satisfying the open set condition. Then, for $\mu$, a  self-affine measure on $X$ (of maximal Hausdorff dimension), we apply \cite[Lemma~6.2]{Fra22} to get
\[
\dim_{\textup{F}}^{1/2} \mu = \frac{\sd (\mu \ast \mu)}{2},
\]
and we can compute the right-hand side explicitly since $\mu \ast \mu$ is again a  self-affine measure and $\min\{\sd (\mu \ast \mu),2\}$ coincides with the $L^2$-dimension, for which there is an explicit formula.   The Fourier spectrum then gives non-trivial information in the sense that it beats the trivial lower bound $\fs X \geq \theta \hd X$ provided $\sd (\mu \ast \mu) > \hd X$  (and this is indeed the case).  This strategy can be employed in much greater generality than we give here, but we leave further examples to the reader; recall Corollary \ref{thm:exceptionalFourierConvolutions} for a general formulation. Better bounds would be obtained by considering smaller $\theta$, but the Fourier spectrum becomes harder to estimate in this case.

Fix $\alpha, \beta, \gamma \in (0,1/3]$ and let $E_\alpha, E_\beta, E_\gamma \subseteq [0,1]$ denote the middle $(1-2\alpha)$,  $(1-2\beta)$ and  $(1-2\gamma)$ Cantor sets, respectively.  Let $X = E_\alpha \times E_\beta\times E_\gamma$, which is a self-affine set of Hausdorff  dimension 
\[
\hd X = \frac{\log 2}{-\log\alpha} + \frac{\log 2}{-\log\beta} + \frac{\log 2}{-\log\gamma}  .
\]
 In many cases it is clear that $\fd X = 0$ (for example if at least one of $\alpha, \beta, \gamma$ is the reciprocal of a Pisot number, see \cite[Theorem 8.3]{Mat15}) and we are unaware of any examples where it is known that $\fd X >0$.  Let $\mu$ be the uniform measure on $X$, that is, the self-affine measure defined  with 8 equal weights.  Then $\mu \ast \mu$ is a self-affine measure associated with a system of 27 maps (not with equal weights) and by standard results (or direct calculation)
\begin{align*}
 \sd (\mu \ast \mu) &= \frac{\log \left( (1/4)^2+(1/2)^2 + (1/4)^2 \right)}{\log \alpha} + \frac{\log \left( (1/4)^2+(1/2)^2 + (1/4)^2 \right)}{\log \beta} \\
 &\quad \quad +   \frac{\log \left( (1/4)^2+(1/2)^2 + (1/4)^2 \right)}{\log \gamma}\\ 
&= \frac{\log (8/ 3)}{-\log \alpha}+ \frac{\log (8/ 3)}{-\log \beta} + \frac{\log (8/ 3)}{-\log \gamma}.
\end{align*}
By Corollary \ref{thm:exceptionalFourierConvolutions} (setting $n=2$) we get
\begin{equation*}
		\hd \{ V\in G(3,k) : \hd  P_{V}(X)<u \}\leq  \max\left\{ 0,  2+2u-\frac{\log (8/ 3)}{-\log \alpha}- \frac{\log (8/ 3)}{-\log \beta} - \frac{\log (8/ 3)}{-\log \gamma} \right\}
	\end{equation*}
for all $u \in [0,k]$ where $k \in \{1,2\}$. See Figure~\ref{fig:boundimprovement} 
for a comparison of the above bound with Peres--Schlag's \eqref{eq:peresschlagbound} and Mattila's \eqref{eq:mattilaboundgeneral} for the cases $k = 1$ and $k=2$.  Interesting to  observe is the fact that in Figure~\ref{fig:boundimprovement} the bound given with the Fourier spectrum shows that dimension of the exceptional set has $0$ Hausdorff dimension in a range of values of $u$ for which the other two bounds give positive dimension.

  \begin{figure}[H]
    \begin{subfigure}{.5\textwidth}
      \centering
      \begin{tikzpicture}[scale = 0.7]
        \def\a{1/3}
        \def\b{1/4}
        \def\c{1/5}
        \def\k{1}
        \def\dimX{ln(2)*(-1/ln(\a) - 1/ln(\b) - 1/ln(\c))}
        \def\half{ln(8/3)*(-1/ln(\a) - 1/ln(\b) - 1/ln(\c))}
        \begin{axis}[
            axis lines = left,
            ymin = 0,
            ymax = 2.1,
            xmin = 0,
            xmax = 1.1,
            xlabel=$u$,
            xtick = {\k}, 
            ytick = {0,1,2},
            xticklabels = {$1$},
            yticklabels = {$0$,$1$,$2$},
            legend style={ at={(2,0.1)},
    anchor=north east,at={(axis description cs:1,0.3)} }
        ]
        \addplot [ 
            domain=0:\k,
            samples=100, 
            color=lightblue,
            style=thick,
            name path=PS
        ]
        {min(\k*(3-\k), max(0, 2-\dimX + x))};
        \addplot [ 
            domain=0:\k,
            samples=100, 
            color=OliveGreen,
            style=thick,
            name path=Mat2
        ]
        {min(\k*(3-\k), max(0, \k*(3-\k-1)+x))};
         \addplot [ 
            domain=0:\k,
            samples=100, 
            color=darkred,
            style=thick,
            name path=FdO
        ]
        {min(\k*(3-\k), max(0, 2 + 2*x - \half))};
      \legend{Peres--Schlag, Mattila $k=1$, Corollary~\ref{thm:exceptionalFouriercoro}}
        \end{axis}
      \end{tikzpicture}
    \end{subfigure}%
    \begin{subfigure}{.5\textwidth}
      \centering
      \begin{tikzpicture}[scale = 0.7]
        \def\a{1/3}
        \def\b{1/4}
        \def\c{1/5}
        \def\k{2}
        \def\dimX{ln(2)*(-1/ln(\a) - 1/ln(\b) - 1/ln(\c))}
        \def\half{ln(8/3)*(-1/ln(\a) - 1/ln(\b) - 1/ln(\c))}
        \begin{axis}[
            axis lines = left,
            ymin = 0,
            ymax = 2.1,
            xmin = 0,
            xmax = 1.7,
            xlabel=$u$,
            xtick = {\dimX}, 
            ytick = {0,2},
            xticklabels = {$\hd X$},
            yticklabels = {$0$, $2$},
            legend style={ at={(2,0.1)},
    anchor=north east,at={(axis description cs:1,0.3)} }
        ]
        \addplot [ 
            domain=0:\dimX,
            samples=100, 
            color=lightblue,
            style=thick,
            name path=PS
        ]
        {min(\k, max(0, 2-\dimX + x))};
        \addplot [ 
            domain=0:\dimX,
            samples=100, 
            color=OliveGreen,
            style=thick,
            name path=Mat2
        ]
        {min(\k, max(0, \k*(3-\k-1)+x))};
         \addplot [ 
            domain=0:\dimX,
            samples=100, 
            color=darkred,
            style=thick,
            name path=FdO
        ]
        {min(\k, max(0, 2 + 2*x - \half))};
      \legend{Peres--Schlag, Mattila $k=2$, Corollary~\ref{thm:exceptionalFouriercoro}}
  
        \end{axis}
      \end{tikzpicture}
    \end{subfigure}
    \caption{Exceptional set estimates given by the Fourier spectrum in red, Peres--Schlag's \eqref{eq:peresschlagbound} in blue and Mattila's \eqref{eq:mattilaboundgeneral} in green, for $\alpha = 1/3$, $\beta = 1/4$ and $\gamma=1/5$, with $\hd X\approx 1.56$. Left: when $k=1$. Right: when $k=2$.}\label{fig:boundimprovement}
    \end{figure}
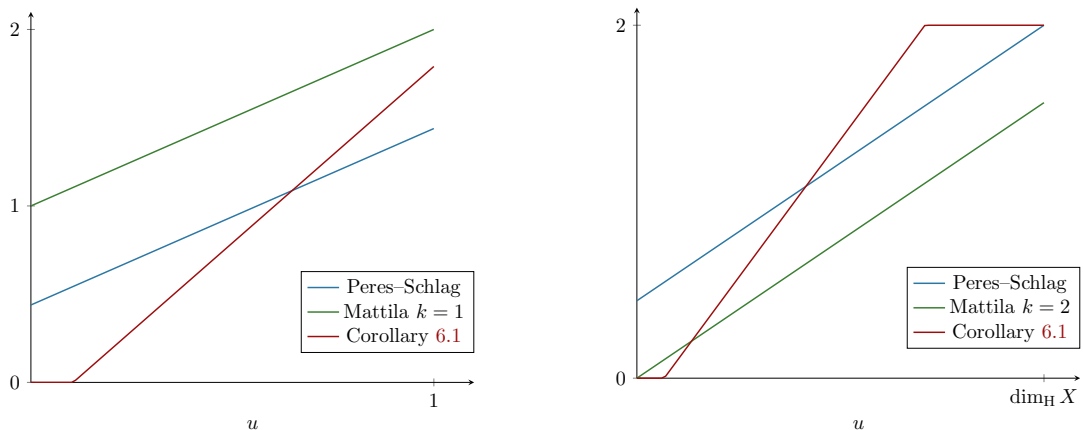

\section*{Acknowledgements} 

We thank Tuomas Orponen,  Alex Rutar and Pablo Shmerkin for helpful conversations and comments.

\end{document}